\documentclass{amsart}
\usepackage{amssymb,amsmath,latexsym}
\usepackage{amsthm}
\usepackage{fontenc}
\usepackage{amssymb}

\numberwithin{equation}{section}

\newtheorem{theorem}{Theorem}[section]

\newtheorem{proposition}[theorem]{Proposition}

\begin{document}
\author{Alexander E Patkowski}
\title{On a solution to a functional equation}

\maketitle
\begin{abstract}We offer a solution to a functional equation using properties of the Mellin transform. A new criteria for the Riemann Hypothesis is offered as an application of our main result, through a functional relationship with the Riemann xi function.\end{abstract}



\section{Introduction} 
It is known that the Mellin transform offers a rich theory in solving differential, integral, as well as difference equations [3]. The approach is used to first reduce these equations to a simpler functional equation, and then apply the inverse transform. \par The purpose of this short note is to study some properties of the solution of the three-variable functional equation
\begin{equation} f(z,y+x)+zf(z,y)=zg(y). \end{equation}
Comparing this equation with the methods outlined in [3, pg.172] in studying difference equations, one can find that we have the solution
\begin{equation}f(z,y)=-\frac{1}{2\pi i}\int_{(a)}g(y+xs)\left(\frac{\pi}{\sin(\pi s)}\right)z^{-s}ds,\end{equation}
where $-1<\Re(s)=a<0.$ To see this, one needs to replace $y$ by $y+x,$ replace $s$ by $s-1,$ moving the line of integration to $-1<\Re(s)-1<0.$ Then compute the residue at the pole $s=0$ to find we arrive at (1.1). To ensure convergence of our integral, we use Stirling's formula [3, pg.121] 
$\sigma<\infty,$ 
$$\Gamma(\sigma\pm it)=O(t^{\sigma-\frac{1}{2}}e^{-\frac{\pi}{2}t}),$$ as $t\rightarrow\infty,$ to make sure that $g$ must not grow as fast as $\pi/\sin(\pi s).$ Meaning, we require
\begin{equation}|g(s)|<Ae^{(r+\delta)|s|},\end{equation}
for constants $A, r>0$ and small $\delta>0.$ This also implies that we have to have $0<x<\pi/r.$

\begin{proposition} Let $g(z)$ be an analytic function for $\Re(z)>0$ which satisfies (1.3) in this region. Let $H(t)$ be a locally integratable function which satisfies the condition $H(t)=t^{-1}H(t^{-1}),$ and has the Mellin transform $g(s)$ for $s\in\mathbb{C}.$ Then we have that 
$$f(z,y)=\frac{1}{x}\int_{0}^{\infty}\frac{zt}{zt+1}t^{-y/x+1/x-1}H(t^{1/x})dt,$$ is a solution to the functional equation (1.1) when $0<x<\pi/r.$
\end{proposition}

\begin{proof} First, it is known [3] that for $-1<\Re(s)<0,$
\begin{equation}\int_{0}^{\infty}t^{s-1}\left(\frac{t}{t+1}\right)dt=-\frac{\pi}{\sin(\pi s)},\end{equation}
and by hypothesis for $s\in\mathbb{C}$
\begin{equation}\int_{0}^{\infty}t^{s-1}H(t)dt=\int_{0}^{\infty}t^{s-2}H(t^{-1})dt=\int_{0}^{\infty}t^{-s}H(t)dt=g(s),\end{equation}
hence $g(s)=g(1-s).$ On the other hand, replacing $s$ by $y+sx$ and changing variable $t$ to $t^{1/x},$ gives
$$g(y+sx)=\frac{1}{x}\int_{0}^{\infty}t^{-y/x-s+1/x-1}H(t^{1/x})dt.$$ (For the existence of such a function $H(t)$ see, for example, [2, Proposition 5.55, pg.152].)
In conjunction with Parseval's theorem for Mellin transforms with (1.4), and noting $\mathbb{C}\cap \{s: -1<\Re(s)<0\}=\{s: -1<\Re(s)<0\}$ is the common region of analyticity, we find the solution offered in our proposition.
\end{proof}

\section{The Riemann xi function}
We provide a new criteria for the Riemann Hypothesis from our proposition from the previous section by appealing to results on the Riemann xi function $\xi(s):=\frac{1}{2}s(s-1)\pi^{-\frac{s}{2}}\Gamma(\frac{s}{2})\zeta(s)$ [4]. As noted in [3], the functions $g$ which satisfy the (1.3) condition in our introduction include entire functions of order one. A well-known example of such a function is the Riemann xi function [4, pg.29]. It is a known result [1, pg.207--208] that for any $s\in \mathbb{C},$ 
\begin{equation}\int_{0}^{\infty}t^{s-1}\bar{H}(t)dt=\xi(s),\end{equation}
where 
\begin{equation}\bar{H}(t):=2t^2\sum_{n\ge1}(2\pi^2 n^4t^2-3\pi n^2)e^{-\pi n^2 t^2},\end{equation} 
for $t>0.$
\begin{theorem} Let $\rho=\frac{1}{2}+i\gamma$ denote the non-trivial zeros of the Riemann zeta function $\zeta(s),$ and $0<x<\pi/r'.$ Then the Riemann hypothesis is equivalent to the statement that $\bar{f}(z,\rho+x)+z\bar{f}(z,\rho)=0,$ where
$$\bar{f}(z,y):=\frac{1}{x}\int_{0}^{\infty}\frac{zt}{zt+1}t^{-y/x+1/x-1}\bar{H}(t^{1/x})dt,$$
and $$|\xi(s)|<A'e^{(r'+\delta)|s|},$$ for constants $A', r'>0.$
\end{theorem}
\begin{proof} From Proposition 1.1 with $g(s)=\xi(s)$ and (2.1)--(2.2), we have that $\bar{f}(z,y+x)+z\bar{f}(z,y)=z\xi(y).$ This together with the fact that the Riemann hypothesis is implied from $\xi(\rho)=0,$ gives the result.
\end{proof}

1390 Bumps River Rd. \\*
Centerville, MA
02632 \\*
USA \\*
E-mail: alexpatk@hotmail.com


\begin{thebibliography}{9}

\bibitem{ConcreteMath} H. M. Edwards. \emph{Riemann's Zeta Function,} 1974. Dover Publications.

\bibitem{ConcreteMath} H. Iwaniec and E. Kowalski, \emph{Analytic number theory,} American Mathematical Society Colloquium Publications, vol. 53, American Mathematical Society, Providence, RI, 2004.


\bibitem{ConcreteMath} R. B. Paris, D. Kaminski, \emph{ Asymptotics and Mellin--Barnes Integrals.} Cambridge University Press. (2001)


\bibitem{ConcereteMath} E. C. Titchmarsh, \emph{The theory of the Riemann zeta function,} Oxford University Press,
2nd edition, 1986.



\end{thebibliography}
\end{document}